\documentclass[a4paper, 11pt]{amsart}
\usepackage[utf8]{inputenc}
\usepackage[T1]{fontenc}
\usepackage{lmodern}
\usepackage[english]{babel}
\usepackage[dvipsnames]{xcolor}
\usepackage{adjustbox}
\usepackage{letltxmacro}
\usepackage{todonotes}
\LetLtxMacro\todonotestodo\todo
\renewcommand{\todo}[2][]{\todonotestodo[#1]{TODO: {#2}}}
\usepackage{enumerate}
\usepackage{hyperref}
\usepackage{url}
\usepackage{mathrsfs}
\usepackage{amssymb}

\PassOptionsToPackage{hyphens}{url}\usepackage{hyperref}

\theoremstyle{definition}
\newtheorem{theorem}{Theorem}

\makeatletter
\newtheorem*{rep@theorem}{\rep@title}
\newcommand{\newreptheorem}[2]{%
\newenvironment{rep#1}[1]{%
 \def\rep@title{#2 \ref{##1}}%
 \begin{rep@theorem}}%
 {\end{rep@theorem}}}
\makeatother

\newreptheorem{theorem}{Theorem}

\newtheorem{lemma}{Lemma}[section]
\newtheorem{proposition}[lemma]{Proposition}
\newtheorem{corollary}[lemma]{Corollary}

\newtheorem{remark}[lemma]{Remark}

\newtheorem*{claim*}{Claim}
\newtheorem{notation}[lemma]{Notation}
\newtheorem*{theorem*}{Theorem}
\newtheorem*{corollary*}{Corollary}
\newtheorem*{lemma*}{Lemma}

\newtheorem*{problem*}{Problem}

\newcommand{\Q}{\mathbb{Q}}
\newcommand{\Z}{\mathbb{Z}}
\newcommand{\N}{\mathbb{N}}

\DeclareMathOperator{\Int}{Int}

\title[Lengths of factorizations of IVP\lowercase{s} on Krull domains with prime elements]{Lengths of factorizations of integer-valued polynomials on Krull domains with prime elements}
\author{Victor Fadinger}
\thanks{V. Fadinger is supported by the Austrian Science Fund (FWF): W1230}

\author{Daniel Windisch}
\thanks{D.~Windisch is supported by the Austrian Science Fund (FWF): P~30934}

\begin{document}

\begin{abstract}
Let $D$ be a Krull domain admitting a prime element with finite residue field and let $K$ be its quotient field. We show that for all positive integers $k$ and $1 < n_1 \leq \ldots \leq n_k$ there exists an integer-valued polynomial on $D$, that is, an element of $\Int(D) = \{ f \in K[X] \mid f(D) \subseteq D \}$, which has precisely $k$ essentially different factorizations into irreducible elements of $\Int(D)$ whose lengths are exactly $n_1,\ldots,n_k$. Using this, we characterize lengths of factorizations when $D$ is a unique factorization domain and therefore also in case $D$ is a discrete valuation domain. This solves an open problem proposed by Cahen, Fontana, Frisch and Glaz.
\end{abstract}

\maketitle

\section{Introduction}

In a foregoing manuscript~\cite{FFW}, Frisch and the present authors investigated factorizations of integer-valued polynomials on valuation rings $V$ of global fields $K$. They showed that the factorization behaviour in this case is as wild as can be: For all positive integers $k$ and all integers $1 < n_1 \leq \ldots \leq n_k$ there exists an integer-valued polynomial $H$ of 
\[\Int(V) = \{f \in K[X] \mid f(V) \subseteq V \}\] 
which has precisely $k$ essentially different factorizations into irreducible elements of $\Int(V)$ whose lengths are exactly $n_1,\ldots,n_k$. This result gave a partial answer to the following open problem:

\begin{problem*}\cite[Problem 39]{open}
Analyze and describe non-unique factorization in $\Int(V)$, where $V$ is a DVR with finite residue field. 
\end{problem*}

The above problem stands in a long tradition of the study of factorizations in integer-valued polynomial rings. Building on results by Cahen and Chabert \cite{elasticitycahen} and Chapman and McClain~\cite{elasticitychapman}, Frisch \cite{integers} showed that every finite multiset of integers $>1$ occurs as the set of lengths (of factorizations into irreducibles) of some polynomial in $\Int(\Z)$. This was generalized to rings of integer-valued polynomials on Dedekind domains with infinitely many maximal ideals of finite index by Frisch, Nakato and Rissner \cite{globalcase} and is analogous to the result on DVRs above~\cite{FFW}.  

Since, so far, rings of integer-valued polynomials are not accessible by a general unified theory of factorizations, as it is the case for Krull domains, one has been forced to take a closer look to special multiplicative properties  of particular (irreducible) elements of these rings.

First steps into the direction of such a unified theory of factorizations for integer-valued polynomials were made by Reinhart~\cite{Reinhart1} and Frisch~\cite{Frisch_mon}, respectively, who proved that $\Int(D)$ is monadically Krull whenever $D$ is a unique factorization domain (UFD), respectively, more generally, a Krull domain. This means that the local multiplicative behaviour in such rings works as in Krull domains. In a follow-up paper, Reinhart~\cite{Reinhart2} then determined class groups of the monadic submonoids of $\Int(D)$. However, in practice these class groups are very hard to compute and therefore have not led to concrete applications.


A special situation occurs, when considering $\Int(D)$, were $D$ is a Krull domain all of whose height one prime ideals have infinite index. In this case $\Int(D)=D[X]$, so $\Int(D)$ is a Krull domain with divisor class group $\mathcal C_v(\Int(D))\cong \mathcal C_v(D)$ and prime divisors in all classes. This implies that the factorization behavior of $\Int(D)$ is determined by the combinatorial object $\mathcal B(\mathcal C_v(D))$, the monoid of zero-sum sequences over the class group. This is a very well investigated object and two cases are to distinguish: 1. $\mathcal C_v(D)$ is infinite: Then $\Int(D)$ has full system of sets of lengths (arbitrary multisets of lengths need not be realizable).
2. $\mathcal C_v(D)$ is finite: For this case a vast amount of literature is available and a lot is known.


In contrast to the effort investigating factorizations in rings of integer-valued polynomials, interestingly even the local case, that is, the case of integer-valued polynomials on DVRs, remained widely open. In the present manuscript, we close this gap and, in particular, solve the open problem above. In fact, we prove a stronger result:

\begin{theorem*}
Let $D$ be a Krull domain admitting a prime element $\pi$ such that $D/\pi D$ is finite. 

Then $\Int(D)$ has full system of multisets of lengths, i.e., for all positive integers $k$ and all integers $1 < n_1 \leq \ldots \leq n_k$ there exists an integer-valued polynomial $H \in \Int(D)$ which has precisely $k$ essentially different factorizations into irreducible elements of $\Int(D)$ whose lengths are exactly $n_1,\ldots,n_k$.
\end{theorem*}

Note that the class of Krull domains admitting a principal prime ideal includes the following: 
\begin{enumerate}
\item unique factorization domains (UFDs) so, in particular,
\item discrete valuation rings,
\item rings of integers in global fields,
\item monoid algebras that are Krull, domains, which include polynomial rings over Krull domains.
\end{enumerate}
 
In the case of UFDs, one can observe a particularly nice dichotomy: Depending on the existence of a principal prime ideal (or, equivalently, height-one prime ideal) with finite residue field, factorizations of elements are either unique or as wild as possible.

\begin{corollary*}
Let $D$ be a unique factorization domain. Exactly one of the following holds:
\begin{itemize}
\item[(1)] $\Int(D) = D[X]$ is a unique factorization domain.
\item[(2)] $\Int(D)$ has full system of multisets of lengths, i.e., for all positive integers $k$ and all integers $1 < n_1 \leq \ldots \leq n_k$ there exists an integer-valued polynomial $H \in \Int(D)$ which has precisely $k$ essentially different factorizations into irreducible elements of $\Int(D)$ whose lengths are exactly $n_1,\ldots,n_k$.
\end{itemize}

Moreover, $(1)$ holds if and only if the residue field of each height-one prime ideal in $D$ is infinite.
\end{corollary*}

\section{Preliminaries}

\textbf{Factorizations.} We give an informal presentation of factorizations. The interested reader is refered to the monograph by Geroldinger and Halter-Koch \cite{GHK} for a systematic introduction.

Let $R$ be an integral domain and $r\in R$. We say that $r$ is \textit{irreducible} (in $R$) if it cannot be written as the product of two nonunits of $R$. A \textit{factorization} of $r$ is a decomposition
\[r=a_1\cdots a_n\]
into irreducible elements $a_i$ of $R$. In this case $n$ is called the \textit{length} of this factorization of $r$. Let $s$ be a further element of $R$. We say that $r$ and $s$ are associated if there exists a unit $\varepsilon\in R$ such that $r=\varepsilon s$. We want to consider factorizations up to order and associates. In other words two factorizations
\[r=a_1\cdots a_n=u_1\cdots u_m\]
of $r$ are \textit{essentially the same} if $n=m$ and, after re-indexing if necessary, $u_i$ is associated to $a_i$ for all $i\in \{1,\ldots ,n\}$. Otherwise, the factorizations are called \textit{essentially different}.\\

\textbf{Valuations.} Let $K$ be a field. A valuation $\mathsf v$ on $K$ is a map
\[\mathsf v:K^\times\to G\]
where $(G,+,\leq)$ is a totally ordered Abelian group, subject to the following conditions for all $a,b\in K^\times$:
\begin{enumerate}
\item $\mathsf v(a\cdot b)=\mathsf v(a) +\mathsf v(b)$ and
\item $\mathsf v(a+b)\geq \inf\{\mathsf v(a),\mathsf v(b)\}$.
\end{enumerate}
The set $\{0\}\cup\{x\in K^\times\mid \mathsf v(x)\geq 0\}$ is called the \textit{valuation ring} of $\mathsf v$. It is a subring of $K$ with quotient field $K$. 

We will often use implicitly the following fact about valuations, which follows from the definition by an easy exercise: If $\mathsf v$ is a valuation on $K$ and $a,b\in K$ are such that $\mathsf v(a)\neq \mathsf v(b)$ then 
$$\mathsf v(a+b)= \inf\{\mathsf v(a),\mathsf v(b)\}.$$

If $\mathsf v(K^\times)\cong \Z$ we call $\mathsf v$ a \textit{discrete valuation} (by the more precise terminology of Bourbaki it would be a \textit{discrete rank one valuation}). If $\mathsf v$ is a discrete valuation on $K$, then there exists a valuation $\mathsf w:K^\times \to \Z$ with $\mathsf w(K^\times)=\Z$ and the same valuation ring as $\mathsf v$. We call $\mathsf w$ the \textit{normalized valuation} of this valuation ring.

For a general introduction to valuations, see \cite{Bourbaki}.\\

\textbf{Newton polygons.} Let $(K,\mathsf v)$ be a discretely valued field, where $\mathsf v$ is normalized, and let $f=a_0+a_1X+\ldots +a_nX^n$ be a polynomial over $K$. The \textit{Newton polygon} of $f$ is contructed as follows. To every term $a_iX^i$ we associate the point $(i,\mathsf v(a_i))\in \mathbb Z^2$, where we just ignore a point if its value in the second coordinate is infinity. We now form the lower boundary of the convex hull of the set of points $\{(0,\mathsf v(a_0)),\ldots,(n,\mathsf v(a_n))\}$ and call it the Newton polygon of $f$. For an introduction to this topic, see \cite[Ch. II, §6]{Neukirch}.\\

\textbf{Integer-valued polynomials.} Let $R$ be an integral domain with quotient field $K$. The set
\[\text{Int}(R)=\{f\in K[X]\mid f(R)\subseteq R\}\] 
is a subring of $K[X]$ and called the \textit{ring of integer-valued polynomials} on $R$.
Let $V$ be the valuation ring of valuation $\mathsf v$ on a field $K$. Every element $f\in K[X]$ can be written in the form $f=\frac{g}{d}$, where $g\in V[X]$ and $d\in V\setminus\{0\}$. It is immediate that $f\in \Int(V)$ if and only if $\min_{a\in V} \mathsf v(f(a))\geq \mathsf v(d)$.

For a detailed treatment of integer-valued polynomials we refer to the monograph by Cahen and Chabert \cite{Cahen-Chabert}.

We need the following fact that seems to be folklore. Lacking a proper reference, we give a proof of the statement.

\begin{remark}\label{remark:infiniteindex}

Let $D$ be an integral domain and $P \subseteq D$ be a prime ideal. Let $F \in D[X] \setminus P[X]$ be a polynomial.

If $\deg F < |D/P|$ then $F(a) \notin P$ for some $a \in D$.
\end{remark}

\begin{proof}
Assume to the contrary that $F(a) \in P$ for all $a \in D$. Then the reduction $\overline{F}$ of $F$ modulo $P$ is a non-zero polynomial that has a zero at every element of $D/P$. This is a contradiction to $\deg \overline{F} \leq \deg F < |D/P|$.

\end{proof}

\section{Construction of irreducible polynomials with prescribed valuations}

In this section, we construct irreducible polynomials whose minimal valuations on a given residue class can be prescribed. We will use those in the proof of the main result to build an integer-valued polynomial that has factorization lengths of our choice.

The following remark seems to be a well-known result, but we could not find a quotable reference, so we included a proof.

\begin{remark}\label{remark:Newton}
Let $(K,\mathsf v)$ be a discretely valued field and let $F$ be a polynomial of degree $n$ over $K$. If the Newton-polygon of $F$ has just one slope $\frac{\lambda}{n}$ such that $\gcd(\lambda,n)=1$, then $F$ is irreducible over $K$.
\end{remark}
\begin{proof}
Let $F$ be as in the statement and assume to the contrary that $F$ is reducible. Let $L$ be a splitting field of $F$ and let $\mathsf w$ be an extension of $\mathsf v$ to $L$. Now let $\alpha\in L$ be a root of $F$, then $[K(\alpha):K]=d<n$ and by \cite[Ch. II, Proposition 6.3]{Neukirch} $\mathsf w(\alpha)=-\frac{\lambda}{n}$. On the other hand, $\mathsf w(\alpha)\in \frac{1}{d}\Z$, so there is $m\in \Z$ with $m=\frac{\lambda d}{n}$, contradicting $\gcd(\lambda, n)=1$. 
\end{proof}

The following fact was proven by Samuel~\cite[Proposition 13]{Samuel} for Noetherian domains and Halter-Koch gave an alternative proof for Dedekind domains, see~\cite[Proposition 2.1]{globalcase}. The proof for Krull domains is identical to the latter but for convenience of the reader, we include it here.

\begin{remark}\label{remark:finitelymany}
Let $D$ be Krull domain and $\Sigma$ a positive integer. Then there are only finitely many height-one prime ideals $Q$ of $D$ with $|D/Q| \leq \Sigma$.
\end{remark}

\begin{proof}
Assume to the contrary that there exist infinitely many height-one prime ideals of $D$ with $|D/Q| = \Sigma$, and let $a \in D \setminus \{0\}$. Then, by definition of a Krull domain, there exist infinitely many height-one primes $Q$ of $D$ such that $|D/Q| = \Sigma$ and $a \notin Q$. For each such $Q$, we obtain $a^{q-1} - 1 \in Q$ because the multiplicative group of the field $D/Q$ has exactly $q-1$ elements. Again by definition of a Krull domain, $a^{q-1} - 1= 0$, whence $a^{q-1} = 1$. So, every non-zero element of $D$ is $(q-1)$-st root of unity, which is a contradiction.
\end{proof}

\begin{lemma}\label{lemma:glueing}
Let $D$ be a Krull domain admitting a prime element $\pi$ such that $D/\pi D$ is finite. Denote by $K$ the quotient field of $D$ and let $\Sigma$ be a positive integer. Let $\mathsf v_Q: K^\times \to \Z$ be the normalized valuation on $K$ with respect to $Q \in \mathfrak{X}(D)$ and denote $\mathsf{v} = \mathsf{v}_{\pi D}$. Let $\pi D = R_1,\ldots,R_q$ be the residue classes with respect to $\pi D$. For each $k \in \{1,\ldots,q\}$ choose $r_k \in R_k$ such that $r_1=\pi$ and $r_s - r_1 \in Q$ for all $s\in \{1,\ldots,q\}$ and all $Q \in \mathfrak{X}(D) \setminus \{\pi D\}$ with $|D/Q| \leq \Sigma$ (this is possible by Remark \ref{remark:finitelymany} and the Approximation Theorem for Krull domains).

Then, for every $s\in \{1,\ldots,q\}$ and every positive integer $n \geq 2$, there exist infinitely many polynomials $F \in D[X]$ of degree $n$ that are pairwise non-associated and irreducible over $K$ with the following properties:
\begin{itemize}
\item[(1)] $\min\{ \mathsf v(F(a)) \mid a\in R_s\}=\mathsf v(F(r_s))=n$.
\item[(2)] For all $a\in D\setminus R_s$ we have $\mathsf v(F(a))=0$.
\item[(3)] $\mathsf{v}_Q(F(0)) = 0$ for all $Q \in \mathfrak{X}(D) \setminus \{ \pi D\}$ with $|D/Q| \leq \Sigma$.
\end{itemize}
\end{lemma}
\begin{proof}
First, let $s=1$, so $R_s=\pi D$. Setting $c = \pi^{n+1}$ and $m \geq n+1$, we claim that $F=X^n + \pi^m X^{n-1} +\sum_{i=0}^{n-2} c X^i$ has the desired properties. Let $a\in \pi D$. Then
\begin{align*}
\mathsf v(F(a))&\geq \min\{\mathsf v(c),\mathsf v(c)+\mathsf v(a),\ldots, m +(n-1)\mathsf v(a), n\mathsf v(a)\}\\
&\geq n=\min\{\mathsf v(c),\mathsf v(c)+\mathsf v(r_1),\ldots, m+(n-1)\mathsf v(r_1), n\mathsf v(r_1)\}=\mathsf v(F(r_1)),
\end{align*}
where the last equality follows from the fact that the minimum is only accepted at the value $n\mathsf v(r_1)$, which proves (1).

For (2), let $a\in D\setminus \pi D$. Then
\[
\mathsf v(F(a))\geq \min\{\mathsf v(c), \mathsf v(c)+\mathsf v(a),\ldots, m+(n-1)\mathsf v(a), n\mathsf v(a)\}=0.
\]
Since this minimum is only accepted at the value $n\mathsf v(a)$, we conclude $\mathsf v(F(a))=0$. 

(3) is trivial as $F(0) = \pi^{n+1}$. To see that $F$ is irreducible, note that its Newton-polygon has exactly one slope, which is $-\frac{n+1}{n}$, and use Remark~\ref{remark:Newton} for the discrete valuation ring $D_{\pi D}$.

The fact that we can vary $m$ over all integers $\geq n+1$ gives infinitely many such $F$ pairwise non-associated over $K$.

Now let $s\neq 1$. We set $F'(X)=F(X-r_s+r_1)$. Since $X\mapsto X-r_s+r_1$ induces an isomorphism $K[X]\to K[X]$, we conclude that all these $F'$ are irreducible and pairwise non-associated over $K$. Moreover,
\[
\mathsf v(F'(r_s))=\mathsf v(F(r_1))=n
\]
and if $a\in R_s$, then $a-r_s\in \pi D$, so $(a-r_s)+r_1\in \pi D$, whence
\[
\mathsf v(F'(a))=\mathsf v(F(a-r_s+r_1))\geq \mathsf v(F(r_1))=n=\mathsf v(F'(r_s)).
\]

On the other hand, if $a\in D\setminus R_s$ then $a-r_s\notin \pi D$, thus $(a-r_s)+r_1\notin \pi D$ and it follows
\[
\mathsf v(F'(a))=\mathsf v(F(a-r_s+r_1))=0.
\]
Finally, $\mathsf v_Q(F'(0)) =\mathsf v_Q(F(r_1-r_s)) =\mathsf v_Q(c) = 0$ for $Q \neq \pi D$ with $|D/Q| \leq \Sigma$.
\end{proof}

\section{Combinatorial toolbox}

The following combinatorial section is identical with the one in \cite{FFW}, but to make the manuscript readable without consulting \cite{FFW} all the time, we include it here, but without proofs.

\begin{notation}\label{notation:interval}
Let $n$ be a positive integer. We write $[n] = \{1,\ldots,n\}$.
\end{notation}

\begin{notation}\label{notation:hyperplanes}
Let $2 \leq k$, $2\leq n_1 \leq \ldots \leq n_k$ be integers, $i,j \in \{1, \ldots,k\}$ with $i < j$, $S \subseteq [n_i]$, and $T \subseteq [n_j]$. We set 
\[ H_{i,j}(S,T) = H_{j,i}(T,S) = [n_1] \times \ldots \times [n_{i-1}] \times S \times [n_{i+1}] \times \ldots \times [n_{j-1}] \times T \times [n_{j+1}] \times \ldots \times [n_k].\]

For $s \in [n_i]$, we define $H_{i,j}(s,T)=H_{i,j}(\{s\},T)$. Moreover, we write $H_{i,j}(S,[n_j]) = H_i(S)$. \\

Note that the $H_i(s)$ are $(k-1)$-dimensional hyperplanes in the grid $[n_1] \times \ldots \times [n_k]$. Analogously, the $H_{i,j}(s,t)$ are $(k-2)$-dimensional hyperplanes.
\end{notation}

\begin{lemma}\label{lemma:hyperplanes}
Let $k >2$ and $1< n_1 \leq \ldots \leq n_k$ be integers. Let $I \subseteq [n_1] \times \ldots \times [n_k]$. Assume that for every $i \in \{1,\ldots,k\}$ and $r \in [n_i]$ there exists $j \in \{1,\ldots,k\}\setminus \{i\}$ and $T \subseteq [n_j]$ such that $I \cap H_i(r) = H_{i,j}(r,T)$. In other words, every intersection of $I$ with a $(k-1)$-dimensional hyperplane is the union of $(k-2)$-dimensional parallel hyperplanes.

Then there exists $\ell \in \{1, \ldots, k\}$ and $S \subseteq [n_\ell]$ such that $I = H_\ell(S)$. That is, $I$ is the union of $(k-1)$-dimensional parallel hyperplanes.
\end{lemma}

\begin{notation}\label{notation:tensors}
Let $k >1$ and $1< n_1 \leq \ldots \leq n_k$ be integers. By $\mathbb{Q}^{n_1\times \ldots \times n_k}$ we denote the set of all $(n_1 \times \ldots \times n_k)$--arrays, that is, the $k$-dimensional analogues of matrices over $\Q$. Let $M \in \mathbb{Q}^{n_1\times \ldots \times n_k}$. For $i \in [n_1] \times \ldots \times [n_k]$, we write $M_i$ for the entry of $M$ indexed by $i$.

For $I \subseteq [n_1] \times \ldots \times [n_k]$, let $Z_I = \{M \in \mathbb{Q}^{n_1\times \ldots \times n_k} \mid \sum_{i \in I} M_i = 0\}$. Moreover,
\[ Z := \bigcap_{\substack{\ell \in \{1,\ldots k\} \\ r \in [n_\ell]}} Z_{H_\ell(r)}.\]

For instance, if $k = 2$ then $Z$ is the set of all $(n_1 \times n_2)$--matrices with all row and column sums equal to $0$.
\end{notation}

Since the elements of $Z$ are defined by the property that sums over hyperplanes are $0$, clearly, sums over disjoint unions of hyperplanes are also $0$. The next lemma shows that no other sum of a subset of the entries of $M \in Z$ is necessarily $0$. We will use it to show the existence of an array $M \in \mathbb{Q}^{n_1\times \ldots \times n_k}$ such that the sum over a subset of the entries of $M$ is $0$ if and only if the corresponding index set is a disjoint union of hyperplanes.

\begin{lemma}\label{lemma:onesum}
Let $k > 1$ and $1 < n_1 \leq \ldots \leq n_k$ be integers. Let $I \subseteq [n_1] \times \ldots \times [n_k]$ be non-empty such that $I \neq H_\ell(S)$ for all $\ell \in \{1,\ldots,k\}$ and $S \subseteq [n_\ell]$.

Then $Z \setminus Z_I \neq \emptyset$.
\end{lemma}

\begin{proposition}\label{proposition:existence}
Let $k$ be a positive integer and $1 < n_1 \leq \ldots \leq n_k$ integers. Then there exists $M \in Z$ such that $M \in Z_I$ only if $I$ is a disjoint union of hyperplanes, that is, $I = H_\ell(S)$ for some $\ell \in \{1,\ldots,k\}$ and $S \subseteq [n_\ell]$.
\end{proposition}


\begin{notation}\label{notation:pairsoftensors}
Let $k,q >1$ and $1< n_1 \leq \ldots \leq n_k$ be integers. For $I_1,\ldots,I_q \subseteq [n_1] \times \ldots \times [n_k]$, let
\[
Z_{I_1,\ldots,I_q} = \{(M_1,\ldots,M_q) \in (\mathbb{Q}^{n_1\times \ldots \times n_k})^q \mid \forall i,j \ \sum_{m \in I_i} (M_i)_m = \sum_{m\in I_j} (M_j)_m\}.
\]

Moreover, we define
\[ \overline Z := \bigcap_{\substack{\ell \in \{1,\ldots k\} \\ r \in [n_\ell]}} Z_{H_\ell(r),\ldots, H_\ell(r)}.\]

For $s\in \Q$ we denote by $(s)$ the array all of whose entries are $s$.
\end{notation}

\begin{lemma}\label{lemma:avoidance}
Let $k,q > 1$ and $1 < n_1 \leq \ldots \leq n_k$ be integers. Let $I_1,\ldots, I_q \subseteq [n_1] \times \ldots \times [n_k]$ be non-empty such that either $I_i\neq I_j$ for some $i\neq j$ or $I_1 \neq H_\ell(S)$ for all $\ell \in \{1,\ldots,k\}$ and $S \subseteq [n_\ell]$.

Then $\overline Z \setminus Z_{I_1,\ldots,I_q} \neq \emptyset$.
\end{lemma}
%

\begin{proposition}\label{proposition:existenceofpairs}
Let $k,q>1$ and $1 < n_1 \leq \ldots \leq n_k$ be integers. Then there exists $(M_1,\ldots,M_q) \in \overline Z$ such that $(M_1,\ldots,M_q) \in Z_{I_1,\ldots,I_q}$ only if $I_1=\ldots=I_q$ is a disjoint union of hyperplanes, that is, $I_1 = H_\ell(S)$ for some $\ell \in \{1,\ldots,k\}$ and $S \subseteq [n_\ell]$. Moreover, the $M_i$ can be chosen such that all of their entries are positive integers.
\end{proposition}

\section{Sets of lengths of integer-valued polynomials over a DVR}

\begin{lemma}\label{lemma:Krullfactorizations}
Let $D$ be a Krull domain with quotient field $K$ and let $I$ be a finite set. Suppose that $D$ admits a prime element $\pi$ and denote by $\mathsf{v}$ its normalized valuation on $K$. Let $f\in \Int(D)$ be of the form
\[
f=\frac{\prod_{i\in I}f_i}{\pi^e},
\]
where, for every $i\in I$, $f_i\in D[X]$ is irreducible over $K$, $e\in \N$, and such that
\[
\min\{\mathsf v_Q(f(a))\mid a\in D\}=0 \text{ for all } Q\in \mathfrak{X}(D).
\]

If $f=g_1\cdots g_m$ is a decomposition into non-units in $\Int(D)$ then every $g_j$ is, up to multiplication by units of $D$, of the form
\[
g_j=\frac{\prod_{i\in I_j}f_i}{\pi^{e_j}},
\]
where $I_1,\ldots,I_m$ are non-empty and form a partition of $I$ and $\sum_{j=1}^m e_j=e$.
\end{lemma}

\begin{proof}
Let $f=g_1\cdots g_m$, then clearly no $g_j$ is a constant, because of the condition on the minima of valuations of evaluations of $f$. From the irreducibility of the $f_i$ it follows that every $g_j=\frac{\prod_{i\in I_j}f_i}{a_j}$ with $a_j\in K$, where the $I_j$ are non-empty and form a partition of $I$ and $\prod_{j=1}^m a_j=\pi^e$. Now it is clear, that $a_j=\pi^{e_j}$, where $\sum_{j=1}^m e_j=e$.
\end{proof}

\begin{lemma}\label{lemma:irreducibles}
Let $D$ be a Krull domain admitting a prime element $\pi$ such that $D/\pi D$ is finite. Let $K$ be the quotient field of $D$ and $\mathsf v: K^\times \to \Z$ the normalized valuation with respect to $\pi$. By $R_1,\ldots,R_q$ denote the residue classes of $D$ modulo $\pi D$. Let $k$ be a positive integer and $1 < n_1 \leq \ldots \leq n_k$ integers.

Let $\Sigma$ be a positive integer and let $r_1,\ldots,r_q$ be a complete system of residues with $r_i \in R_i$ such that $r_1=\pi$ and $r_s - r_1 \in Q$ for all $s\in \{1,\ldots,q\}$ and all $Q \in \mathfrak{X}(D) \setminus \{\pi D\}$ with $|D/Q| \leq \Sigma$ (this is possible by Remark \ref{remark:finitelymany} and the Approximation Theorem for Krull domains).

For $i \in [n_1] \times \ldots \times [n_k]$ and $s \in \{1, \ldots, q\}$, let $F_i^{(s)} \in D[X]$ irreducible over $K$ such that
\begin{itemize}
\item[(i)] $\min \{\mathsf{v}(F_i^{(s)}(a)) \mid a \in R_s\} = \mathsf{v}(F_i^{(s)}(r_s))$ for all $i \in [n_1] \times \ldots \times [n_k]$ and $s \in \{1, \ldots, q\}$,
\item[(ii)] $\min \{\mathsf{v}(F_i^{(s)}(a)) \mid a \in R_t\} = 0 = \mathsf{v}(F_i^{(s)}(r_t))$ for all $i \in [n_1] \times \ldots \times [n_k]$, $s \in \{1, \ldots, q\}$ and $t \in \{1, \ldots, q\} \setminus \{s\}$,
\item[(iii)] $\sum_{i \in [n_1] \times \dots \times [n_k]} \mathsf{v}(F_i^{(s)}(r_s)) = \sum_{i \in [n_1] \times \dots \times [n_k]} \mathsf{v}(F_i^{(t)}(r_t))$ for all $s,t \in \{1,\ldots,q\}$,
\item[(iv)] for each $Q \in \mathfrak{X}(D) \setminus \{\pi D\}$ there exits $a \in D$ such that $\mathsf{v}_Q(F_i^{(s)}(a)) = 0$ for all $i \in [n_1] \times \ldots \times [n_k]$ and $s \in \{1, \ldots, q\}$, and
\item[(v)] $\sum_{i,s} \deg F_i^{(s)} \leq \Sigma$.
\end{itemize}
Define $e = \sum_{i \in [n_1] \times \dots \times [n_k]} \mathsf{v}(F_i^{(s)}(r_s))$ for some $s \in \{1,\ldots,q\}$ (this is independent from $s$ by (iii)) and set 
\begin{align*}
H = \frac{\prod_{s = 1}^q \prod_{i \in [n_1] \times \ldots \times [n_k]} F_i^{(s)}}{\pi^e}.
\end{align*}

Then $H \in \Int(D)$.
Furthermore, $H$ is a product of two non-units in $\Int(D)$ if and only if there exist non-empty $I_1,\ldots,I_q \subsetneqq [n_1] \times \ldots \times [n_k]$ such that $\sum_{i \in I_s} \mathsf{v}(F_i^{(s)}(r_s)) = \sum_{i \in I_t} \mathsf{v}(F_i^{(t)}(r_t))$ for all $s,t \in \{1,\ldots,q\}$.

\end{lemma}

\begin{proof}
Clearly $H$ is integer-valued over $D$. If there exist $I_1,\ldots,I_q$ as in the lemma, define $e' = \sum_{i \in I_s} \mathsf{v}(F_i^{(s)}(r_s))$ for some $s \in \{1,\ldots,q\}$. Then clearly 
\begin{align*}
H = \frac{\prod_{s = 1}^q \prod_{i \in I_s} F_i^{(s)}}{\pi^{e'}} \cdot \frac{\prod_{s = 1}^q \prod_{i \in ([n_1] \times \ldots \times [n_k]) \setminus I_s} F_i^{(s)}}{\pi^{e-e'}}
\end{align*}
is a decomposition into two non-units in $\Int(D)$.

Conversely, if $H = H_1 \cdot H_2$ is a decomposition of $H$ where $H_1$ and $H_2$ are non-units of $\Int(D)$, then, by Lemma~\ref{lemma:Krullfactorizations}, there exist non-empty $I_1,\ldots,I_q \subsetneqq [n_1]\times \ldots \times [n_k]$  such that 
\[ H_1 = \frac{\prod_{s = 1}^q \prod_{i \in I_s} F_i^{(s)}}{\pi^{e'}} \hspace{1cm}\text{ and }\hspace{1cm} H_2 = \frac{\prod_{s = 1}^q \prod_{i \in ([n_1] \times \ldots \times [n_k]) \setminus I_s} F_i^{(s)}}{\pi^{e-e'}} \]
for some $e' \in \{0,\ldots,e\}$ (note that the property (ii) in Lemma~\ref{lemma:Krullfactorizations} is satisfied by (iv) and (v) together with Remark~\ref{remark:infiniteindex}). Assume to the contrary that $\sum_{i \in I_s} \mathsf{v}(F_i^{(s)}(r_s)) \neq \sum_{i \in I_t} \mathsf{v}(F_i^{(t)}(r_t))$ for some $s,t \in \{1,\ldots,q\}$, say $\sum_{i \in I_s} \mathsf{v}(F_i^{(s)}(r_s)) > \sum_{i \in I_t} \mathsf{v}(F_i^{(t)}(r_t))$. Since $H_1$ is an integer-valued polynomial on $V$, it follows that 
\[  \sum_{i \in I_s} \mathsf v(F_i^{(s)}(r_s)) >  \sum_{i \in I_t} \mathsf v(F_i^{(s)}(r_t)) \geq e'.\] 
Hence we get 
\[  \sum_{i \in ([n_1]\times \ldots \times [n_k]) \setminus I_s} \mathsf v(F_i^{(s)}(r_s)) <  \sum_{i \in ([n_1]\times \ldots \times [n_k]) \setminus I_t} \mathsf v(F_i^{(s)}(r_t)) \leq e - e', \]
which is a contradiction because $H_2 \in \Int(V)$. 
\end{proof}

By a \textit{$q$-tuple of ordered partitions} of sets $I_1,\ldots,I_q$, we mean an equivalence class of tuples $((I_1^{(1)},\ldots,I_1^{(\ell)}),\ldots,(I_q^{(1)},\ldots,I_q^{(\ell)}))$, where the $I_s^{(\lambda)}$ form a partition of $I_s$, under the equivalence relation where
\[
((I_1^{(1)},\ldots,I_1^{(\ell)}),\ldots,(I_q^{(1)},\ldots,I_q^{(\ell)}))
\]
is identified with
\[
((I_1^{(\sigma(1))},\ldots,I_1^{(\sigma(\ell))}),\ldots,(I_q^{(\sigma(1))},\ldots,I_q^{(\sigma(\ell))}))
\]
for all permutations $\sigma$ of $\{1,\ldots,\ell\}$. The positive integer $\ell$ is called the \textit{length} of the $q$-tuple of ordered partitions.

\begin{lemma}\label{lemma:factorizations}

Let $D$ be a Krull domain admitting a prime element $\pi$ such that $D/\pi D$ is finite. Let $K$ be the quotient field of $D$ and $\mathsf v: K^\times \to \Z$ the normalized valuation with respect to $\pi$. By $R_1,\ldots,R_q$ denote the residue classes of $D$ modulo $\pi D$. Let $k$ be a positive integer and $1 < n_1 \leq \ldots \leq n_k$ integers.

Let $\Sigma$ be a positive integer and let $r_1,\ldots,r_q$ be a complete system of residues with $r_i \in R_i$ such that $r_1=\pi$ and $r_s - r_1 \in Q$ for all $s\in \{1,\ldots,q\}$ and all $Q \in \mathfrak{X}(D) \setminus \{\pi D\}$ with $|D/Q| \leq \Sigma$ (this is possible by Remark \ref{remark:finitelymany} and the Approximation Theorem for Krull domains).

For $i \in [n_1] \times \ldots \times [n_k]$ and $s \in \{1, \ldots, q\}$, let $F_i^{(s)} \in V[X]$ irreducible over $K$ such that
\begin{itemize}
\item[(i)] $\min \{\mathsf{v}(F_i^{(s)}(a)) \mid a \in R_s\} = \mathsf{v}(F_i^{(s)}(r_s))$ for all $i \in [n_1] \times \ldots \times [n_k]$ and $s \in \{1, \ldots, q\}$,
\item[(ii)] $\min \{\mathsf{v}(F_i^{(s)}(a)) \mid a \in R_t\} = 0 = \mathsf{v}(F_i^{(s)}(r_t))$ for all $i \in [n_1] \times \ldots \times [n_k]$, $s \in \{1, \ldots, q\}$ and $t \in \{1, \ldots, q\} \setminus \{s\}$,
\item[(iii)] $\sum_{i \in [n_1] \times \dots \times [n_k]} \mathsf{v}(F_i^{(s)}(r_s)) = \sum_{i \in [n_1] \times \dots \times [n_k]} \mathsf{v}(F_i^{(t)}(r_t))$ for all $s,t \in \{1,\ldots,q\}$,
\item[(iv)] for each $Q \in \mathfrak{X}(D) \setminus \{\pi D\}$ there exits $a \in D$ such that $\mathsf{v}_Q(F_i^{(s)}(a)) = 0$ for all $i \in [n_1] \times \ldots \times [n_k]$ and $s \in \{1, \ldots, q\}$, and
\item[(v)] $\sum_{i,s} \deg F_i^{(s)} \leq \Sigma$.
\end{itemize}
Define $e = \sum_{i \in [n_1] \times \dots \times [n_k]} \mathsf{v}(F_i^{(s)}(r_s))$ for some $s \in \{1,\ldots,q\}$ (this is independent from $s$ by (iii)) and set 
\begin{align*}
H = \frac{\prod_{s = 1}^q \prod_{i \in [n_1] \times \ldots \times [n_k]} F_i^{(s)}}{\pi^e}.
\end{align*}

Then, for every $\ell>0$, a bijective correspondence between, on the one hand, $q$-tuples of ordered partitions
\[ ((I_1^{(1)},\ldots,I_1^{(\ell)}),\ldots,(I_q^{(1)},\ldots,I_q^{(\ell)})) \]
of $[n_1]\times \ldots \times [n_k]$ satisfying
\begin{itemize}
\item[(1)] $e_j := \sum_{i \in I_s^{(j)}} \mathsf{v}(F_i^{(s)}(r_s)) = \sum_{i \in I_t^{(j)}} \mathsf{v}(F_i^{(t)}(r_t))$ for all $s,t \in \{1,\ldots,q\}$ and $j \in \{1, \ldots, \ell\}$,
\item[(2)] for all $j \in \{1, \ldots, \ell\}$ and non-empty $J_1^{(j)} \subsetneqq I_1^{(j)},\ldots, J_q^{(j)} \subsetneqq I_q^{(j)}$ there exist $s,t \in \{1,\ldots,q\}$ such that $\sum_{i \in J_s^{(j)}} \mathsf{v}(F_i^{(s)}(r_s)) \neq \sum_{i \in J_t^{(j)}} \mathsf{v}(F_i^{(t)}(r_t))$,
\end{itemize}
and, on the other hand, essentially different factorizations of $H$ into $\ell$ irreducible elements of $\Int(D)$ is given by
\begin{align*}
((I_1^{(1)},\ldots,I_1^{(\ell)}),\ldots,(I_q^{(1)},\ldots,I_q^{(\ell)})) \mapsto  \frac{\prod_{s=1}^q \prod_{i \in I_s^{(1)}} F_i^{(s)}}{\pi^{e_1}} \cdot \ldots \cdot \frac{\prod_{s=1}^q \prod_{i \in I_s^{(\ell)}} F_i^{(s)}}{\pi^{e_\ell}}.
\end{align*} 
\end{lemma}

\begin{proof}
This follows immediately from Lemma~\ref{lemma:irreducibles}.
\end{proof}

\begin{theorem}\label{theorem:setsoflenghts}
Let $D$ be a Krull domain admitting a prime element $\pi$ such that $D/\pi D$ is finite. 

Then $\Int(D)$ has full system of multisets of lengths, i.e., for all positive integers $k$ and all integers $1 < n_1 \leq \ldots \leq n_k$ there exists an integer-valued polynomial $H \in \Int(D)$ which has precisely $k$ essentially different factorizations into irreducible elements of $\Int(D)$ whose lengths are exactly $n_1,\ldots,n_k$.
\end{theorem}

\begin{proof}
Let $K$ be the quotient field of $D$, $\mathsf v: K^\times \to \Z$ the normalized valuation of the prime element $\pi \in D$. By $\pi D = R_1,\ldots,R_q$ denote the residue classes of $D$ modulo $\pi D$.

By Proposition~\ref{proposition:existenceofpairs}, there exists a $q$-tuple of arrays $(M_1,\ldots,M_q) \in (\Q^{n_1\times \ldots \times n_k})^q$ where all arrays have positive integers as entries and such that $(M_1,\ldots,M_q) \in \overline Z$ with $(M_1,\ldots,M_q) \in Z_{I_1,\ldots,I_q}$ only if $I_1=\ldots=I_q$ is a disjoint union of hyperplanes, that is, $I_1 = H_\ell(S)$ for some $\ell \in \{1,\ldots,k\}$ and $S \subseteq [n_\ell]$ (see Notation~\ref{notation:hyperplanes}, \ref{notation:tensors} and \ref{notation:pairsoftensors}). Recall that $(M_s)_i$ for $i \in [n_1]\times \ldots \times [n_k]$ denotes the $i$-th entry of $M_s$.

We set $\Sigma=\sum_{i\in [n_1] \times \ldots \times [n_k]} (M_s)_i$ which is independent of the choice of $s$. Let $r_1,\ldots,r_q$ be a complete system of residues with $r_i \in R_i$ such that $r_1=\pi$ and $r_s - r_1 \in Q$ for all $s\in \{1,\ldots,q\}$ and all $Q \in \mathfrak{X}(D) \setminus \{\pi D\}$ with $|D/Q| \leq \Sigma$ (this is possible by Remark \ref{remark:finitelymany} and the Approximation Theorem for Krull domains).

If $k=1$ then $H= X^{n_1}$ has the desired property. So let $k\geq 2$. By Lemma~\ref{lemma:factorizations} it suffices to construct polynomials $F_i^{(s)}$ as in the hypothesis of this lemma such that there are exactly $k$ different $q$-tuples of partitions of $[n_1] \times \ldots \times [n_k]$ and the cardinalities of the partitions are exactly $n_1,\ldots,n_k$.

By Lemma~\ref{lemma:glueing}, there exist, for $i \in [n_1]\times \ldots \times [n_k]$ and $s \in \{1,\ldots,q\}$, polynomials $F_i^{(s)}\in D[X]$ which are all pairwise non-associated and irreducible over $K$ with the following properties:
\begin{itemize}
\item[(i)] $\min \{\mathsf{v}(F_i^{(s)}(a)) \mid a \in R_s\} = \mathsf{v}(F_i^{(s)}(r_s)) = (M_s)_i=\deg(F_i^{(s)})$ for all $i \in [n_1] \times \ldots \times [n_k]$ and $s \in \{1, \ldots, q\}$,
\item[(ii)]  $\min \{\mathsf{v}(F_i^{(s)}(a)) \mid a \in R_t\} = 0 = \mathsf{v}(F_i^{(s)}(r_t))$ for all $i \in [n_1] \times \ldots \times [n_k]$, $s \in \{1, \ldots, q\}$ and $t \in \{1, \ldots, q\} \setminus \{s\}$, and
\item[(iii)] $\mathsf{v}_Q(F_i^{(s)}(0)) = 0$ for all $Q \in \mathfrak{X}(D) \setminus \{ \pi D\}$ with $|D/Q| \leq \Sigma$ and all $i \in [n_1] \times \ldots \times [n_k]$ and $s \in \{1, \ldots, q\}$.
\end{itemize}

By the choice of the $M_s$, the only admissible (in the sense of Lemma~\ref{lemma:factorizations}) $q$-tuples of ordered partitions of the index set $[n_1]\times \ldots \times [n_k]$, namely those that correspond to factorizations into irreducibles, are the ones of the form $((H_r(1),\ldots, H_r(n_r)), \ldots, (H_r(1), \ldots, H_r(n_r)))$ for $r \in \{1,\ldots,k\}$. These are exactly $k$ many of cardinalities $n_1,\ldots,n_k$.
\end{proof}

\begin{corollary}
Let $D$ be a unique factorization domain. Exactly one of the following holds:
\begin{itemize}
\item[(1)] $\Int(D) = D[X]$ is a unique factorization domain.
\item[(2)] $\Int(D)$ has full system of multisets of lengths, i.e., for all positive integers $k$ and all integers $1 < n_1 \leq \ldots \leq n_k$ there exists an integer-valued polynomial $H \in \Int(D)$ which has precisely $k$ essentially different factorizations into irreducible elements of $\Int(D)$ whose lengths are exactly $n_1,\ldots,n_k$.
\end{itemize}

Moreover, $(1)$ holds if and only if the residue field of each height-one prime ideal in $D$ is infinite.
\end{corollary}

\begin{proof}
By~\cite[Corollary I.3.15]{Cahen-Chabert}, $\Int(D) = D[X]$ if and only if the residue field of each height-one prime ideal in $D$ is infinite, in which case $\Int(D)$ is a UFD. In the other case, use Theorem~\ref{theorem:setsoflenghts}.
\end{proof}

\bibliographystyle{amsplainurl}
\bibliography{bibliography}

\vspace{0.5cm}
\noindent
\textsc{Victor Fadinger, Institute for Mathematics and Scientific Computing, Universität Graz, Heinrichstraße 36, 8010 Graz, Austria} \\
\textit{E-mail address}: \texttt{viktor.fadinger@gmail.com}\\

\noindent
\textsc{Daniel Windisch, Department of Analysis and Number Theory (5010), Technische Universität Graz, Kopernikusgasse 24, 8010 Graz, Austria} \\
\textit{E-mail address}: \texttt{dwindisch@math.tugraz.at}

\end{document}